\documentclass[11pt,draft]{article}
 \usepackage{amsfonts,latexsym,amsmath,amssymb,amsthm}
\usepackage[mathscr]{eucal}
\usepackage{graphicx,color}
\usepackage{comment}
\usepackage[hidelinks]{hyperref}
\usepackage{epstopdf}
\usepackage[margin = 2cm]{geometry}
\usepackage[inline]{enumitem}
\usepackage{mathtools}
\usepackage{cite}
\usepackage{xcolor}
\usepackage{cleveref}
\usepackage{tikz}
\usepackage{array}
\usepackage{calc}
\usepackage[normalem]{ulem}
\usepackage{stmaryrd}
\usepackage{bbm}

\usepackage{setspace}
\usetikzlibrary{arrows, arrows.meta}

\newcommand{\pvi}[1]{
\left\{\begin{array}{lcc}
#1
\end{array}
\right.
}
\newtheorem{theorem}{Theorem}
\newtheorem{lemma}[theorem]{Lemma}
\newtheorem{proposition}[theorem]{Proposition}
\newtheorem{corollary}[theorem]{Corollary}
\theoremstyle{definition}

\newtheorem{remarkx}[theorem]{Remark}

\newenvironment{remark}
  {\pushQED{\qed}\remarkx}
  {\popQED\endremarkx}

\DeclarePairedDelimiter{\norm}{\lVert}{\rVert}
\DeclarePairedDelimiter{\abs}{\lvert}{\rvert}

\newcommand{\suchthat}{\ifnum\currentgrouptype=16 \mathrel{}\middle|\mathrel{}\else\mid\fi}

\def\R{{\mathbb R}}    % field of real number
\makeatletter
\newcommand{\customlabel}[2]{%
   \protected@write \@auxout {}{\string \newlabel {#1}{{#2}{\thepage}{#2}{#1}{}} }%
   \hypertarget{#1}{#2}
}

\makeatother

\begin{document}

\setlist[enumerate]{label={(\alph*)}}
\setlist[enumerate, 2]{label={(\alph{enumi}-\roman*)}}
\newlist{listhypo}{enumerate}{1}
\setlist[listhypo]{label={\textup{(H\arabic*)}}, ref={\textup{(H\arabic*)}}}

\title{\bf Stability analysis of a linear system coupling wave and heat equations with different time scales\footnotemark[1]}

\author{Gonzalo Arias\footnotemark[2], Eduardo Cerpa\footnotemark[3], \hspace{.1cm }and Swann Marx\footnotemark[4]}

\maketitle

\footnotetext[1]{This work has been partially supported by ANID Millennium Science Initiative
Program trough Millennium Nucleus for Applied Control and Inverse Problems NCN19-161 and STIC-Amsud C-CAIT.
}
\footnotetext[2]{Facultad de Matem\'aticas, Pontificia Universidad Católica de Chile, Avda.
Vicu\~na Mackenna 4860, Macul, Santiago, Chile. E.mail: ngonzaloandres@uc.cl }

\footnotetext[3]{Instituto de Ingenier\'ia Matem\'atica y Computacional, Facultad de Matem\'aticas, Pontificia Universidad Católica de Chile, Avda.
Vicu\~na Mackenna 4860, Macul, Santiago, Chile. E.mail: eduardo.cerpa@uc.cl}

\footnotetext[4]{LS2N, \'Ecole Centrale de Nantes \& CNRS UMR 6004, F-44000 Nantes, France. E.mail: swann.marx@ls2n.fr}

\maketitle
%\abstract{}
\begin{abstract}
In this paper we consider a system coupling a wave equation with a heat equation through its boundary conditions. The existence of a small parameter in the heat equation, as a factor multiplying the time derivative, implies the existence of different time scales between the constituents of the system. This suggests the idea of applying
a singular perturbation method to study stability properties. In fact, we prove that this method works for the system under study. Using this strategy, we get the stability of the system and a Tikhonov theorem, which allows us to approximate the solution of the coupled system using some appropriate uncoupled subsystems.
\end{abstract}

\section{Introduction}

%adding some background on heat-wave systems
There are many models in science and engineering coupling heat equations together with wave equations. For instance, in \cite{green1991re} the authors provide a model for thermoelasticity of this nature, which is called a thermoelasticity model of type II. These models have been studied in past years, see \cite{han2017decay} and the references therein for results regarding the asymptotic behaviour for this systems. On the other hand, and as it is stated in \cite{biot1941general}, the physical process of an earthquake is modelled by a
three dimensional nonlinear coupled heat-wave PDE. A one dimensional simplification of this earthquake model is considered in \cite{gutierrez2022advances}, where the authors implement a sliding mode controller to avoid instabilities, i.e., avoiding earthquakes, and whose design is based on the diffusion process. Also, fluid-structure interaction can be modeled as a coupled system of this type. We refer to the article \cite{MR2289863}, and the references therein, for a deep understanding of the asymptotic behaviour of heat-wave system arising in fluid-structure interaction.
%Introducing singular perturbation and time scales techniques

For coupled systems it may happen that one constituent is better described by using a different time scale. For example, in  \cite[Example 11.1]{khalil2002nonlinear}, it is shown that a model for electrical motors have this property, where the electrical component is faster than the mechanical one, and both phenomena are described by ODEs. Regarding the PDE case, when using the Saint-Venant-Exner equations to study the transport of a sediment in a flow on a reach, it may happen that the sediment dynamic is slower than the flow dynamic (see for instance \cite{hudson2003formulations} or \cite[Section 1.5]{bible_coron}), yielding a two time scale phenomenon.

To address the stability and control design for finite dimensional systems, singular perturbation and separation of time scales have been widely developed in the last decades, see for instance \cite{kokotovic1968singular,kokotovic1972singular,kokotovic1975singular,kokotovic1999singular}. ODEs are finite dimensional systems due to the presence of only one variable, singular perturbation problems for ODEs are regarded generally in the case when one has a coupled ODE system, and one of the components admits a different time scale.

Singular perturbation problems for PDEs are more general, since the singularity may arise by perturbing a high order derivative. For example, if the perturbation parameter appears as a factor multiplying the second order terms of an elliptic operator, and when considering this parameter to vanish, the operator degenerates into a non-elliptic one, see for instance, a survey \cite{kadalbajoo2003singularly} about singular perturbation problems for PDEs. 

In the study of coupled systems with different time scales, and at a first glance, one can think that the fast dynamic may only be important for a short period of time. This is the cornerstone of the singular perturbation theory for systems with different time scales, which we will refer to, simply, as singular perturbation method (SPM, for short). This theory, roughly speaking, states in which situations the fast dynamic can be replaced by some limit process, neglecting the effects of the fast time scale. A general theory for nonlinear finite-dimensional systems is developed in \cite[Chapter 1]{kokotovic1999singular} and \cite[Chapter 11]{khalil2002nonlinear}, where the answer to this issue is given by Tikhonov's theorem. The main idea of the SPM is to decouple a system when one of the equations is fast enough. The decoupling gives two approximated subsystems, each of them, capturing the slow and fast behaviour coming from the full system. The approximated slow system is called \textit{reduced order system}, and after the use of a suitable time-rescaling of the original system, we obtain the approximated fast system, which is called \textit{boundary-layer system}. In recent years, the SPM has been applied to different infinite-dimensional coupled systems with different time scales. Among those, we mention \cite{cerpa2019singular,cerpa-prieur2017,tang2017stability,arias2023frequency} for systems coupling an ODE with either a hyperbolic system or a wave equation, \cite{tang2015singular,Tang2019Singular,tang2015tikhonov} for hyperbolic systems with different time scales and \cite{marx2022singular} for systems coupling a KdV equation with an ODE. And, as it is shown in \cite{tang2015tikhonov,cerpa-prieur2017,tang2017stability}, this technique may fail for some infinite-dimensional systems.

The purpose of this paper is to study the stability properties of a system composed by a heat equation coupled with a wave equation through its boundary conditions, when the dynamics have different time scales. This is the first article addressing this issue for a system coupling parabolic and hyperbolic PDEs. Being precise, let us consider the following system in singularly perturbed form
\begin{equation}
\label{system}
\begin{cases}
    u_{tt}  = u_{xx}, & (x,t)\in (0,1)\times (0,\infty),  \\
    u(0,t) = 0, & t\in (0,\infty),\\
    u_x(1,t) = - a u_t(1,t) + b p(0,t), & t\in (0,\infty),\\
    u(x,0) = u_0(x),\: u_t(x,0) = u_1(x), & x\in (0,1),\\
    \varepsilon p_t  = p_{xx},& (x,t)\in (0,1)\times (0,\infty),\\
    p_x(0,t) = c p(0,t),& t\in (0,\infty),\\
    p_x(1,t) = d u(1,t),& t\in (0,\infty),\\
    p(x,0) = p_0(x),& x\in (0,1),
\end{cases}
\end{equation}
where $a,b,c,d\in \mathbb{R}$, $\varepsilon>0$. Up to our knowledge, we do not know any result regarding the well-posedness and the asymptotic behaviour of \eqref{system}. In order to do so, we employ a semigroup approach to show that \eqref{system} is well-posed on some appropriate spaces, provided that the initial condition $(u_0,u_1,p_0)$ are regular enough. Then, we
analyse the asymptotic behaviour of solutions of \eqref{system} by an energy approach. Therefore, we show that for any $\varepsilon>0$, the full system \eqref{system} is exponentially stable. This results are proven by imposing some restrictions on the parameters $a,b,c,d\in \mathbb{R}$. Also, and since we are interested in the case where the constituents equations of \eqref{system} have different time scales, we are going to suppose that $\varepsilon$ is small enough. Therefore, we are going to apply the SPM to get the stability of \eqref{system} through its subsystems, and give a Tikhonov's approximation result.

As explained before, Tikhonov's approximation has to be done by using the reduced order and boundary-layer systems given by the method. For \eqref{system}, the reduced order system defined by
\begin{equation}
\begin{cases}
\bar u_{tt}  = \bar u_{xx},&  (x,t)\in(0,1)\times(0,\infty),  \\
     \bar u(0,t)  = 0,&  t\in (0,\infty) \\
     \bar u_x(1,t) = - a \bar u_t(1,t) + \frac{bd}{c} \bar u(1,t),&  t\in (0,\infty)\\
     \bar u(x,0) = \bar u_0(x),\: \bar u_t(x,0) = \bar u_1(x),&  x\in(0,1),
\end{cases}
\label{eq:reduced}
\end{equation}
with initial conditions $\bar u_0\in H^1(0,1)$ such that $u_0(0)=0$ and $\bar u_1\in L^2(0,1)$. On the other hand, the boundary-layer system, in the time variable $\tau=\tfrac{t}{\varepsilon}$, is defined through
\begin{equation}
\begin{cases}
    \bar p_\tau  = \bar p_{xx}, & (x,\tau)\in (0,1)\times (0,\infty), \\
    \bar p_x(0,\tau) = c \bar p(0,\tau), & \tau\in(0,\infty),\\
    \bar p_x(1,\tau) = 0,& \tau\in(0,\infty),\\
    \bar p(x,0) = \bar p_0(x),& x\in(0,1),
\end{cases}
\label{boundary_layer}
\end{equation}
with initial condition $\bar p_0 \in L^2(0,1)$. In Section \ref{Stab}, we will derive those subsystems.

Now, we are able to state the main results of this article, the first one being the stability of system \eqref{system}, independent of the value of $\varepsilon>0$. 
\begin{theorem}
\label{claim_stab}
Let $\varepsilon>0$, and $a,b,c,d \in \mathbb{R}$ satisfying
\begin{itemize}
    \item[(i)] $a\in[\eta^{-1},\eta]\backslash \{1\}$ where $\eta=(\sqrt{3}-1)(\sqrt{3}+1)^{-1} $,
    \item[(ii)] $\abs{\tfrac{bd}{c}}\leq 1$,
    \item[(iii)] $c\geq \tfrac{\pi^2}{8}$,
    \item[(iv)] $1+2(\sinh{\mu}+2\cosh{\mu})b^2\leq c$, where $\mu \leq \mu^\ast= \ln{\left(\tfrac{\abs{1+a}}{\sqrt{2}\abs{1-a}}\right)}$,
    \item[(v)] $\abs{d} \leq 2\sqrt{\mu e^{-\mu}}$, for $\mu \leq \mu^\ast$.
\end{itemize}
There exists a unique mild solution $(u,u_t,p)\in C([0,\infty); H^1(0,1)\times L^2(0,1)\times L^2(0,1))$ to \eqref{system} provided that $(u_0,u_1,p_0)\in H^1(0,1)\times L^2(0,1)\times L^2(0,1)$. Moreover, if $d,c$ satisfies $c\geq \frac{\pi^2}{4}$ and $\abs{d}\leq \sqrt{\mu e^{-\mu}}$, then the origin of \eqref{system} is exponentially stable in the space $H^1(0,1)\times L^2(0,1) \times L^2(0,1)$.
\end{theorem}

\noindent This result is proven in Section \ref{analysis}, when analyzing the well-posedness and exponential stability of \eqref{system}.

\begin{remarkx}
\label{qmu}
The quantity $q(a,\mu):= -\frac32 e^{\mu}(1-a)^2 + \frac 12 e^{-\mu}(1+a)^2$ appears in the proof of Theorem \ref{claim_stab}, and it is non-negative whenever $\mu\leq \mu^\ast$ and $a\in \left[\eta^{-1},\eta\right]\backslash \{1\}$. Moreover, $q(a,\mu)=0$ if and only if $\mu=\mu^\ast$.
\end{remarkx}

\begin{remark}
Even if Theorem \ref{claim_stab} shows the exponential stability of \eqref{system}, the SPM allows us to unify the stability analysis in the space $H^1(0,1)\times L^2(0,1)\times L^2(0,1)$ in Theorem \ref{full_stab} with the
Tikhonov result in Theorem \ref{thikonov}. This mean that we can approximate the system \eqref{system} by the subsystems \eqref{eq:reduced} and \eqref{boundary_layer}, whenever $\varepsilon>0$ is small enough. Moreover, the SPM gives a stronger exponential stability result for \eqref{system}, in the sense that the $L^2(0,1)-$norm of the derivative $p_x(\cdot, x,\varepsilon)$ also decays. Therefore, the origin of \eqref{system} is exponentially stable in $H^1(0,1)\times L^2(0,1)\times H^1(0,1)$. This stability result is not as direct as the property stated in Theorem \ref{claim_stab}.
\end{remark}

\begin{theorem}
\label{full_stab}
Let $a,b,c,d$ satisfy $(i)-(v)$ in Theorem \ref{claim_stab}. Suppose that $b,c,d\in \R$ are such that $\frac{\abs{bd}}{c} \leq \sqrt{\mu e^{-\mu}(F(\mu))^{-1}}, $ where $F(\mu)=2\sinh{(\mu)}+10\cosh{(\mu)}$.
\begin{itemize}
    \item[(i)] If $b,c\in\mathbb{R}$ satisfy $\frac{\pi^2}{4}+F(\mu)b^2\leq c$, then there exists $\varepsilon^\ast_1,C_1>0$ such that for any $\varepsilon\in(0,\varepsilon^\ast_1)$, and for all $(u_0,u_1,p_0)\in H^1(0,1)\times L^2(0,1)\times L^2(0,1)$, the solution 
    to \eqref{system} satisfies
    \begin{equation}
    \label{l2-decay}
    \norm{(u(t),u_t(t),p(t))}_{H^1(0,1)\times L^2(0,1)\times L^2(0,1)} \leq C_1 e^{-\frac{\mu}{4} t} \norm{(u_0,u_1,p_0)}_{H^1(0,1)\times L^2(0,1)\times L^2(0,1)}, \hspace{.2cm} \forall t\geq 0.
    \end{equation}
    \item[(ii)]If $b,c\in\mathbb{R}$ satisfy $3F(\mu)b^2\leq c$.
    Then there exists $\varepsilon^\ast_2,C_2>0$ such that for any $\varepsilon\in(0,\varepsilon^\ast_2)$ and for all $(u_0,u_1,p_0)\in H^1(0,1)\times L^2(0,1)\times H^1(0,1)$, the solution 
    to \eqref{system} satisfies
    \begin{equation}
    \label{h1-decay}
    \norm{(u(t),u_t(t),p(t))}_{H^1(0,1)\times L^2(0,1)\times H^1(0,1)} \leq C_2 e^{-\frac{\mu}{4} t} \norm{(u_0,u_1,p_0)}_{H^1(0,1)\times L^2(0,1)\times H^1(0,1)}, \hspace{.2cm} \forall t\geq 0.
\end{equation}
\end{itemize}
\end{theorem}
\begin{remarkx}
Note that the condition $\frac{\pi^2}{4}+F(\mu)b^2\leq c$ of (i) in Theorem \ref{full_stab} is more restrictive than the fourth one in Theorem \ref{claim_stab}. On the other hand, $3F(\mu)b^2\leq c$, which is the condition $(ii)$ of Theorem \ref{full_stab}), together with $c\geq \frac{\pi^2}{8}$, implies condition (iv) in Theorem \ref{full_stab}. Moreover, the inequality $\frac{\abs{bd}}{c} \leq \sqrt{\mu e^{-\mu}(F(\mu))^{-1}}<1$ is fulfilled for any $\mu>0$.
\end{remarkx}
\begin{theorem}
\label{thikonov}
Let $a,b,c,d$ satisfy $(i)-(v)$ in Theorem \ref{claim_stab}. Suppose that $b,c,d\in \R$ satisfy
\begin{itemize}
    \item[(1)] $c\geq \frac{\pi^2}{4}$,
    \item[(2)] $\frac{\pi^2+2}{4} + F(\mu)b^2 \leq c$,
    \item[(3)] $\abs{d} \leq \sqrt{\mu e^{-\mu}}$.
\end{itemize}
There exists $\varepsilon^\star >0$ such that for any $\varepsilon\in (0,\varepsilon^\star)$, $u_0\in H^1(0,1)$, $u_1\in L^2(0,1)$, $ p_0\in L^2(0,1)$ $\bar u_0\in H^1(0,1)$, $\bar u_1\in L^2(0,1)$, $\bar p_0\in H^1(0,1)$ satisfying the compatibility conditions $u_0(0)=0,\bar u_0(0)=0$, together with the smallness conditions
\begin{equation}
\label{small-difference}
\norm{u_0-\bar u_0}_{H^1(0,1)} + \norm{u_1-\bar u_1}_{L^2(0,1)} + \norm{p_0-\bar p_0-d\tfrac{1+cx}{c}u_0(1)}_{L^2(0,1)} = O(\varepsilon^{3/2}),
\end{equation}
and
\begin{equation}
\label{small-sub}
 \norm{\bar p_0}_{H^1(0,1)} = O(\varepsilon) , \hspace{.3cm} \norm{\bar u_0}_{H^1(0,1)} + \norm{\bar u_1}_{L^2(0,1)} = O(\varepsilon^{3/2}),
\end{equation}
then the unique mild solution to \eqref{system} 
$$(u,u_t,p)\in C([0,\infty); H^1(0,1)\times L^2(0,1)\times L^2(0,1)),$$ 
satisfies
\begin{equation*}
\norm{u(t)-\bar u(t)}_{H^1(0,1)} + \norm{u_t(t)-\bar u_t(t)}_{L^2(0,1)} = e^{-\tfrac{\mu}{8}t}  O(\varepsilon^{3/2}), \hspace{.3cm} \forall t\geq 0,  
\end{equation*}
and
\begin{equation*}
\norm{p(t)-\bar p(\tfrac{t}{\varepsilon})- \tfrac{1+cx}{c} d\bar u_t(1,t)}_{L^2(0,1)}=e^{-\tfrac{\mu}{8}t}  O(\varepsilon), \hspace{.3cm} \forall t\geq 0,
\end{equation*}
where $\bar u\in C([0,\infty);H^1(0,1))\cap C^1([0,\infty);L^2(0,1))$ is the unique mild solution of system \eqref{eq:reduced} and $\bar p\in C([0,\infty);H^1(0,1))$ is the unique mild solution of system \eqref{boundary_layer}. 
\end{theorem}
\begin{remarkx}
Conditions $(1),(2),(3)$ in Theorem \ref{thikonov} and the regularity $\bar p_0\in H^1(0,1)$ imply conditions $(iii)-(v)$ in Theorem \ref{claim_stab} together with the regularity $\bar p_0\in L^2(0,1)$. The regularity of the mild solution to \eqref{boundary_layer} when $\bar p_0 \in H^1(0,1)$ is explained in Section \ref{BoundaryStab}.
\end{remarkx}
The paper is organized as follows. In Section \ref{well-posed} we prove the well-posedness of system \eqref{system} through semigroup theory, while in Section \ref{stab_cond} we show that system \eqref{system} is exponentially stable for any $\varepsilon>0$, by using an energy approach. In Section \ref{ReducedStab} we study the stability of the reduced order system \eqref{eq:reduced} through a suitable Lyapunov functional. In Section \ref{BoundaryStab} we study the stability of the boundary-layer system \eqref{boundary_layer}, for which we characterize the exponential decay in both $L^2(0,1)$ and $H^1(0,1)$ norms, by using two different Lyapunov functionals. In Section \ref{FullStab} we proof Theorem \ref{full_stab}, which shows that the full system is exponentially stable in the $H^1(0,1)\times L^2(0,1)\times L^2(0,1)$ and $H^1(0,1)\times L^2(0,1)\times H^1(0,1)$ spaces, for $\varepsilon>0$ small enough. In Section \ref{Thiko} we proof the Tikhonov approximation result given in Theorem \ref{thikonov}, which justifies the application of the SPM to get this exponential stability result.
\section{Well-posedness and stability for any $\varepsilon$}
\label{analysis}
\subsection{Well-posedness}
\label{well-posed} 
In this section we study the well-posedness of \eqref{system} through semigroup theory. To this end, let us define $H= H^1(0,1)\times L^2(0,1)$ equipped with the following inner product
\begin{equation}
\left((u,v),(\tilde u,\tilde v)\right)_{H} = 2\int_0^1 e^{\mu x}(u_x+v)(\tilde u_x +\tilde v) + e^{-\mu x}(u_x-v)(\tilde u_x-\tilde v)dx,
\end{equation}
which turns out to be a Hilbert space. Therefore, the space $\mathcal{H} =  H \times L^2(0,1)$ is a Hilbert space with the following induced inner product
\begin{equation*}
\left((u,v,p),(\tilde u,\tilde v,\tilde p)\right)_{\mathcal{H}} = \left((u,v),(\tilde u,\tilde v)\right)_H + \varepsilon \int_0^1  p \tilde pdx,
\end{equation*}
which is equivalent to the usual product norm in $\mathcal{H}$. We define the spatial operator associated with \eqref{system}
\begin{equation}
\label{defop}
A_\varepsilon(u,v,p)=\left(v,u_{xx},\varepsilon^{-1} p_{xx}\right),
\end{equation}
with domain
\begin{multline}
\label{domain}
D(A_\varepsilon)=\Big\{(u,v,p)\in \mathcal{H} \text{ : } A_\varepsilon(u,v,p)\in \mathcal{H}, u(0)=v(0)=0,  u_{x}(1)=-av(1)+bp(0), \\ p_{x}(0)=cp(0), p_x(1)=du(1)\Big\}.
\end{multline}
Moreover, we define the space $H^1_L(0,1) =\{ u \in H^1(0,1) \text{ : } u(0)=0\}$, which is equipped with the following inner product  
$$(u,v)_{H^1_L(0,1)}= (u_x,v_x)_{L^2(0,1)},$$ 
The space $H^1_L(0,1)$ turns out to be a Hilbert space with this inner product, thanks to Poincaré's inequality.

\begin{remarkx}
\label{char}
Since $A_\varepsilon$ is linear, we have for each $\varepsilon>0$ the following characterization for its domain:
\begin{equation*}
\begin{aligned}
    D(A_\varepsilon)=\Big\{(u,v,p)\in H^2(0,1)\cap H^1_L(0,1) \times H^1_L(0,1) \times H^2(0,1) & \text{ : } u_x(1)=-av(1)+bp(0), \\& \hspace{.3cm} p_x(0)=cp(0), \hspace{.3cm} p_x(1)=du(1)\Big\}.
\end{aligned}
\end{equation*}
\end{remarkx}
With the aid of the elements introduced so far we have that \eqref{system} is equivalent to
\begin{equation}
\label{sgs}
\pvi{ \dot Y(t)  = A_\varepsilon Y(t), \\ Y(0)=Y_0,}
\end{equation}
where $Y(t)=(u(t,\cdot),v(t,\cdot),p(t,\cdot))$ and $Y_0=(u_0,u_1, p_0) \in \mathcal{H}$. The following result is the main tool to show the well-posedness of the Cauchy problem \eqref{sgs}.
\begin{lemma}
\label{maximal_monotone}
Let $a\in[\eta^{-1},\eta]$, $a\not=1$, $\mu\in (0,\mu^\ast(a)]$ where $\mu^\ast(a),\eta$ are given in Lemma \ref{claim_stab}.
Then for each $\varepsilon>0$ the operator defined in \eqref{defop} is a maximal dissipative operator whenever $c\geq \tfrac{\pi^2}{8}$, $\abs{d} \leq 2\sqrt{\mu e^{-\mu}}$, $\abs{\tfrac{bd}{c}}<1$ and  $1+2(\sinh{\mu}+2\cosh{\mu})b^2 \leq c$.
\end{lemma}
\begin{proof}
First we will prove that $A_\varepsilon$ is a dissipative operator. For this purpose, let $(u,v,p)\in D(A_\varepsilon)$. Integration by parts yields
\begin{multline*}
\left(A_\varepsilon(u,v,p), (u,v,p)\right)_{\mathcal{H}} = -\int_0^1  \abs{p_x}^2 dx + p_x(1)p(1)-p_x(0)p(0) -\mu \int_0^1 e^{\mu x} \abs{u_x+v}^2+e^{-\mu x}\abs{u_x-v}^2dx,\\+e^{\mu}\abs{u_x(1)+v(1)}^2-e^{-\mu}\abs{u_x(1)-v(1)}^2-\abs{u_x(0)+v(0)}^2-\abs{u_x(0)-v(0)}^2,
\end{multline*}
using boundary conditions we have
\begin{multline}
\label{in}
\left(A_\varepsilon(u,v,p),(u,v,p)\right)_{\mathcal{H}} 
=  -\int_0^1  \abs{p_x}^2 dx + du(1)p(1)-c\abs{p(0)}^2  -\mu\int_0^1 e^{\mu x} \abs{u_x+v}^2+e^{-\mu x}\abs{u_x-v}^2dx \\ +2\sinh{(\mu)}b^2\abs{p(0)}^2+\left(e^{\mu}\abs{1-a}^2-e^{-\mu}\abs{1+a}\right)\abs{v(1)}^2  +2e^{\mu}(1-a)v(1)bp(0)+2e^{-\mu}(1+a)v(1)bp(0).
\end{multline}
Applying Young's and Poincaré's inequalities to \eqref{in} we get
\begin{multline*}
\left(A_\varepsilon(u,v,p),(u,v,p)\right)_{\mathcal{H}}\leq  -\int_0^1  \abs{p_x}^2 dx + \frac{\delta }{2}d^2\abs{u(1)}^2+\frac{1}{2\delta}\abs{p(1)}^2-c\abs{p(0)}^2 , \\-\mu\int_0^1 e^{\mu x} \abs{u_x+v}^2+e^{-\mu x}\abs{u_x-v}^2dx +\left(e^{\mu}\abs{1-a}^2-e^{-\mu}\abs{1+a}\right)\abs{v(1)}^2 ,\\+2\sinh{(\mu)}b^2\abs{p(0)}^2 + \frac{1}{2}e^{\mu}(1-a)^2\abs{v(1)}^2+ \frac{1}{2}e^{-\mu}(1+a)^2\abs{v(1)}^2+4 \cosh{(\mu)}b^2 \abs{p(0)}^2.
\end{multline*}
Equivalently, we have
\begin{multline}
\label{delta_ineq}
\left(A_\varepsilon(u,v,p),(u,v,p)\right)_{\mathcal{H}} \leq  -\int_0^1  \abs{p_x}^2 dx + \frac{\delta }{2}d^2\abs{u(1)}^2+\frac{1}{2\delta}\abs{p(1)}^2 -\mu\int_0^1 e^{\mu x} \abs{u_x+v}^2+e^{-\mu x}\abs{u_x-v}^2dx,\\ +(2(\sinh{(\mu)}+2\cosh{(\mu)})b^2-c)\abs{p(0)}^2 -q(a,\mu)\abs{v(1)}^2,
\end{multline}
where $a,\mu$ and $q(a,\mu)$ are choosen as in Theorem \ref{claim_stab}. Now, taking $\delta = \frac{4\mu e^{-\mu}}{d^2}$ in \eqref{delta_ineq}, and using Lemma \eqref{trace_reduced} yields the following inequality
\begin{equation*}
\left(A_\varepsilon(u,v,p),(u,v,p)\right)_{\mathcal{H}} 
\leq  -\int_0^1  \abs{p_x}^2 dx +(2(\sinh{(\mu)}+2\cosh{(\mu)})b^2-c)\abs{p(0)}^2+\frac{d^2e^{\mu}}{8\mu}\abs{p(1)}^2.
\end{equation*}
Using Lemma \ref{a1} we obtain
\begin{equation}
\label{dissipative}
\begin{aligned} \left(A_\varepsilon(u,v,p),(u,v,p)\right)_{\mathcal{H}} &\leq \left(\frac{d^2e^\mu}{4\mu}-1\right)\int_0^1\abs{p_x}^2 dx+ \left( b^2(2\sinh{\mu}+4\cosh{\mu})-c+\frac{d^2e^\mu}{4\mu}\right)\abs{p(0)}^2.
\end{aligned}
\end{equation}
Finally, we deduce from \eqref{dissipative} that $A_\varepsilon$ is a dissipative operator for every $\varepsilon>0$, as soon as
\begin{equation}
\label{c1}
\abs{d} \leq 2 \sqrt{\mu e^{-\mu}}, \hspace{.3cm} 1+b^2(2\sinh{\mu}+4\cosh{\mu})-c\leq 0,
\end{equation}
is satisfied.

On the other hand, since $\rho(A_\varepsilon)$ is open (see, for instance, \cite[Remark 2.2.8]{tucsnak2009observation}), we have that $0\in\rho(A_\varepsilon)$ implies that $\lambda I + A_\varepsilon$ is onto for some $\lambda>0$. Hence, the rank condition given in \cite[Chapter 7]{brezis2010functional} is fulfilled, and the operator $A_\varepsilon$ is a maximal operator. The condition $0\in \rho(A_\varepsilon)$ is equivalent to show that for each $(u_0,v_0,p_0)\in \mathcal{H}$ the equation $A_\varepsilon(u,v,p)=(u_0,v_0,p_0)$ has a solution $(u,v,p)\in D(A_\varepsilon)$, i.e., we have to solve the boundary value problem
\begin{equation}
\label{notmaximality}
\pvi{ u_{xx}= v_0,\\ \varepsilon^{-1} p_{xx} = p_0, \\ u(0)=0, \hspace{.3cm} u_x(1)=-au_0(1)+bp(0), \\ p_x(0)=c p(0), \hspace{.3cm} p_x(1)=du(1).}
\end{equation}
Under the invertible transformation $\tilde p (x) = p(x)-\varepsilon\left(\frac{d}{2c}x^2+\frac{d}{6}x^3\right)u_0(1)-\left(d x+\frac{d}{c}\right)u(1),$ system \eqref{notmaximality} becomes
\begin{equation}
\label{maximality}
\pvi{ u_{xx}= v_0,\\ \varepsilon^{-1} \tilde p_{xx} = p_0-(\tfrac{d}{c}+dx)u_0(1),\\ u(0)=0, \hspace{.3cm} u_x(1)=-au_0(1)+\tfrac{bd}{c} u(1) + b\tilde p(0), \\ \tilde p_x(0)=c \tilde p(0), \hspace{.3cm} \tilde p_x(1)=-\varepsilon(\tfrac{d}{c}+\tfrac{d}{2})u_0(1).}
\end{equation}
This system is a coupled system in cascade form, and therefore is easier to solve. We will use a variational approach to study \eqref{maximality}. Let us multiply the first line and second line of \eqref{maximality} by $\psi\in H^1(0,1)$ and $\phi\in  H^1_L(0,1)$, respectively. Integrating over $(0,1)$, and integrating by part yields
\begin{equation}
\label{variational1}
\varepsilon^{-1}\int_0^1 \tilde p_x \psi_x dx +\varepsilon^{-1}c\tilde p(0) \psi(0)  = \int_0^1 \big((\tfrac{d}{c}+dx)u_0(1)-p_0 \big)\psi dx - \varepsilon^{-1}\tfrac{d(c+2)}{2c}u_0(1)\psi(1), \hspace{.3cm}\forall \psi \in H^1(0,1),
\end{equation}
and
\begin{equation}
\label{variational2}
\int_0^1 u_x\phi_x dx - \frac{bd}{c}u(1)\phi(1) = -\int_0^1 v_0\phi dx + \big(b\tilde p(0)-au_0(1)\big) \phi(1), \hspace{.3cm} \forall \phi\in H^1_L(0,1).
\end{equation}
The right-hand side of both \eqref{variational1} and \eqref{variational2} are continuous with respect its respective arguments $\psi\in H^1(0,1)$ and $\phi\in H^1_L(0,1)$.
The variational problems \eqref{variational1} and \eqref{variational2} have the operators $B_\varepsilon:H^1(0,1)\times H^1(0,1) \to \R$ and $ C:H^1_L(0,1)\times H^1_L(0,1) \to \R$ as bilinear forms. These bilinear forms are defined by
$$  B_\varepsilon(\varphi,\psi)=\varepsilon^{-1} \int_0^1\varphi_x\psi_x dx + \varepsilon^{-1}c\varphi(0)\psi(0), \hspace{.5cm}C(\varphi,\psi) = \int_0^1\varphi_x\psi_x dx  - \tfrac{bd}{c} u(1)\phi(1).$$
It is easy to prove that $B_\varepsilon$ and $C$ are continuous and coercive operators whenever $\abs{\tfrac{bd}{c}}\leq 1$. In fact, let us note that
$$ B_\varepsilon(\psi,\psi) = \varepsilon^{-1}\int_0^1 \abs{\psi_x}^2 + \varepsilon^{-1}c \abs{\psi(0)}^2, \hspace{.3cm} \forall \psi\in H^1(0,1),$$
and
$$C(\varphi,\varphi) = \int_0^1 \abs{\varphi_x}^2 - \frac{bd}{c} \abs{\varphi(1)}^2, \hspace{.3cm} \forall \varphi\in H^1_L(0,1).$$
Applying Lemma \ref{trace_poincare} to $\varphi$, and the fact that $\norm{\varphi}_{H^1(0,1)} \leq \kappa \norm{\varphi}_{L^2(0,1)}$, for some constat $\kappa >0$, it follows directly that $C$ is coercive whenever $\abs{\tfrac{bd}{c}}<1$. On the other hand, using Poincaré's inequality (see for instance \cite[Appendix A]{krstic2008boundary}) for $\psi$ we have
\begin{equation*}
\begin{aligned}
B_\varepsilon(\psi,\psi) &= \varepsilon^{-1} \left( (c-\tfrac{\pi^2}{8})\abs{\psi(0)}^2+\frac{\pi^2}{8}\abs{\psi(0)}^2 + \frac12\int_0^1 \abs{\psi_x}^2dx+ \frac12\int_0^1 \abs{\psi_x}^2dx\right)\geq \frac{\varepsilon^{-1}}{2} \norm{\psi}_{H^1(0,1)},
\end{aligned}
\end{equation*}
whenever $c\geq\frac{\pi^2}{8}$. Therefore, \cite[Corollary 5.8]{brezis2010functional} applies, and both \eqref{variational1} and \eqref{variational2} have one solution $\tilde p \in H^1(0,1)$ and $ u\in H^1_L(0,1) $. It only remains to show that $\tilde p, u \in H^2(0,1)$ and that the boundary conditions are satisfied. To this end, it is enough to take $\varphi,\psi \in C^\infty(0,1)$ and integrate by parts. Then we choose suitable conditions for $\phi,\psi$, e.g., $\varphi(1)=0$, $\psi(0)=0$, and so on.
\begin{proposition}
Let $a,b,c,d$ satisfying the hypothesis of Proposition \ref{maximal_monotone} and $\varepsilon>0$, then we have
\begin{itemize}
\item[(i)] Let $(u_0,u_1,p_0)\in D(A_\varepsilon)$. Then, system \eqref{system} has a unique strong solution
\begin{equation*}
(u,u_t,p)\in C([0,\infty);D(A_\varepsilon))\cap C^1([0,\infty);\mathcal{H}).
\end{equation*}
\item[(ii)] Let $(u_0,u_1,p_0)\in \mathcal{H}$. Then, system \eqref{system} has a unique mild solution $(u,u_t,p)\in C([0,\infty);\mathcal{H})$.
\end{itemize}
\end{proposition}
\begin{proof}
This result is a consequence of using semigroup theory developed in \cite[Chapter 3]{brezis2010functional} together with Lemma \ref{maximal_monotone}
\end{proof}
\begin{remarkx}
Let $a,b,c,d$ satisfying the hypothesis of Proposition \ref{maximal_monotone} and $\varepsilon >0$. Recalling the definitions of $D(A_\varepsilon)$ and $\mathcal{H}$, it holds
\begin{itemize}
    \item[(i)] For $(u_0,u_1,p_0)\in D(A_\varepsilon)$, the unique strong solution $(u,p)$ of \eqref{system} satisfies $$u\in C([0,\infty);H^2(0,1)\cap H^1(0,1))\cap C^1([0,\infty);H^1(0,1))\cap C^2([0,\infty); L^2(0,1))$$ and $p\in C([0,\infty);H^2(0,1))\cap C^1([0,\infty);L^2(0,1))$.
    \item[(ii)] For $(u_0,u_1,p_0)\in H^1(0,1)\times L^2(0,1)\times L^2(0,1)$ the unique mild solution $(u,p)$ of \eqref{system} satisfies $u\in C([0,\infty); H^1(0,1))\cap C^1([0,\infty);L^2(0,1))$ and $p\in C([0,\infty);L^2(0,1))$.
\end{itemize}
\end{remarkx}
\end{proof}
\subsection{Stability conditions}
\label{stab_cond}
In this section we show that the state $(u(\cdot,t),u_t(\cdot,t),p(\cdot,t))$ of \eqref{system} is exponentially stable in the $H^1(0,1)\times L^2(0,1) \times L^2(0,1)-$norm.

\begin{proposition}
Let $a,b,c,d$ satisfying the hypothesis of Proposition \ref{maximal_monotone} and $\varepsilon>0$. Therefore,
\begin{multline}
\label{dif_ineq}
\dot E(u,u_t,p) \leq - \frac{\mu}{2} \norm{(u,u_t)}_H  - \frac{\pi^2}{8}\norm{p}^2_{L^2(0,1)} + \left(\frac{d^2e^{\mu}}{2\mu} -\frac{1}{2}\right)\int_0^1 \abs{p_x}^2dx ,\\ +\left(\frac{\pi^2}{8}-\frac{c}{2}\right)\abs{p(0,t)}^2  + \left( \frac{d^2e^{\mu}}{2\mu}+(2\sinh{\mu} + 4\cosh{\mu})b^2-\frac{c}{2}\right)\abs{p(0,t)}^2,
\end{multline}
\end{proposition}
\begin{proof}
For $Y=(u,u_t,p)$ strong solution to \eqref{sgs}, let us consider the following energy functional 
\begin{equation*}
E(u,u_t,p)  = \frac12 \norm{(u,u_t,p)}^2_{\mathcal{H}}.
\end{equation*}
Taking the derivative of $E$ along the strong solutions of \eqref{sgs} we get that
\begin{equation*}
\dot E (u,u_t,p) = \left(\dot Y(t),Y(t)\right)_{\mathcal{H}} = \left(A_\varepsilon Y(t),Y(t)\right)_{\mathcal{H}}.
\end{equation*}
Following the computations for the dissipativity of \eqref{defop}, and taking $\delta = \frac{4\mu e^{-\mu}}{d^2}$ in \eqref{delta_ineq} we get that
\begin{multline*}
\dot E(u,u_t,p) \leq - \frac{\mu}{2} \norm{(u,u_t)}_H  -\int_0^1  \abs{p_x}^2 dx +\frac{d^2e^{\mu}}{4\mu}\abs{p(1,t)}^2-c\abs{p(0,t)}^2 ,\\ +2\sinh{\mu}b^2\abs{p(0,t)}^2 + 4\cosh{\mu}b^2\abs{p(0,t)}^2.
\end{multline*}
Using Lemma \ref{wirtinger} we get the desired inequality.
\end{proof}

Now we are able to prove Theorem \ref{claim_stab}.

\begin{proof}[\textbf{Proof of Theorem \ref{claim_stab}}]
Suppose that $a,b,c,d\in \R$ satisfies conditions $(i)-(v)$ of Theorem \ref{claim_stab}. Note that, conditions $(i)-(iii)$ satisfy the hypothesis of Proposition \ref{maximal_monotone}. On the other hand, taking $c\geq \tfrac{\pi^2}{4}$, $1+2(\sinh{(\mu)}+2\cosh{(\mu)})b^2 \leq c$ and $d\leq  \sqrt{\mu e^{-\mu}}$, wich satisfies conditions $(iii),(iv),(v)$ from Theorem \ref{claim_stab}, respectively. Therefore, from \eqref{dif_ineq}, we get that $\dot E(u,u_t,p) \leq -\min\{\tfrac{\mu}{2}, \frac{\pi^2}{8}\}E(u,u_t,p)$. As a consequence the origin of \eqref{system} is exponentially stable for the $H^1(0,1)\times L^2(0,1)\times L^2(0,1)-$norm.
\end{proof}
\section{Stability analysis for small $\varepsilon$}
\label{Stab}
The goal of this section is to prove Theorem \ref{full_stab} by studying the reduced order system and the boundary-layer system, which are given when applying SPM to \eqref{system}. Roughly speaking, Theorem \ref{full_stab} states that the full system is stable whenever the subsystems are exponentially stable, and the parameter $\varepsilon>0$ is small enough. First, let us compute the equilibrium point of $p$ when $\varepsilon=0$. That is, we have to solve the following boundary value problem
\begin{equation*}
\begin{cases}
    p_{xx}(x,t) = 0, & (x,\tau)\in (0,1)\times (0,\infty), \\
    p_x(0,t) = c p(0,t), & t\in(0,\infty),\\
    p_x(1,t) = d u(1,t),& t\in(0,\infty),\\
\end{cases}
\end{equation*}
Directly integrating $p_{xx}(x,t) = 0$ between $0$ and $x$ two times and using the boundary conditions we have that $p(x,t) = \left(\frac{d}{c}  + x d\right) u(1,t)$ is the unique solution, which is considered as the quasi steady-state. We will use this equilibrium point to deduce the the approximated subsystems.
\subsection{Reduced order system}
\label{ReducedStab}
To find the reduced order system we have to set $\varepsilon=0$. As it was shown before, we have that $p(0,t) = \frac{d}{c}u(1,t)$ is defined for all $t\geq 0$, which is the trace of the equilibrum point computed early. Replacing the latter trace equality into the wave equation of \eqref{system}, we have that the  reduced order system, whose state is denoted by $\bar u$, is given by
\begin{equation}
\label{reduced}
\begin{cases}
    \bar u_{tt}  = \bar u_{xx},&  (x,t)\in(0,1)\times(0,\infty),  \\
     \bar u(0,t)  = 0,&  t\in (0,\infty) \\
     \bar u_x(1,t) = - a \bar u_t(1,t) + \frac{bd}{c} \bar u(1,t),&  t\in (0,\infty)\\
     \bar u(x,0) = \bar u_0(x), \bar u_t(x,0) = \bar u_1(x),&  x\in(0,1).
\end{cases}
\end{equation}
We want to study the asymptotic behaviour of \eqref{reduced}. For any $(\bar u_0,\bar u_1)\in H^2(0,1) \times H^1(0,1)$ such that the compatibility conditions $\bar u_0 (0)=\bar u_1(0) =0$ and $ (\bar u_0)_x(1)=-a\bar u_0(1)+\tfrac{bd}{c}\bar u_0(1)$ are satisfied, there exists a unique strong solution of \eqref{reduced} (See for instance \cite[Chapter 7]{komornik1997exact})
$$(\bar u, \bar u_t)\in C([0,\infty);H^2(0,1)\cap H^1(0,1)\times H^1(0,1)) \cap C^1([0,\infty);H^1(0,1)\times L^2(0,1)).$$ 
The mild solution to \eqref{reduced} is defined thorugh an extension by continuity of the semigroup that gives us this solution.

For $\bar u$ the unique strong solution of \eqref{reduced}, we define the Lyapunov functional
\begin{equation}
\label{v1}
V_1(\bar u, \bar u_t) = \frac{1}{2}\int_0^1 e^{\mu x} (\bar u_t+\bar u_x)^2+e^{-\mu x} (\bar u_t-\bar u_x)^2 dx,
\end{equation}
which was first introduced in \cite{smyshlyaev2010boundary}, and used in \cite{cerpa2019singular} when applying the SPM to a system coupling a slow ODE together a fast wave equation.
The following result states the stability of \eqref{reduced} around the origin
\begin{proposition}
\label{reduced_stab}
Let $\mu^\ast$ be given as in Theorem \ref{claim_stab} and $a\in \R$ satisfying $(i)$. Then, there exists a constant $\kappa=\kappa(\mu)>0$ such that the unique mild solution $\bar u\in C([0,\infty);H^1(0,1))\cap C^1([0,\infty); L^2(0,1)) $ of \eqref{reduced} satisfies
\begin{equation*}
\norm{(\bar u(t),\bar u_t(t))}_{H^1(0,1)\times L^2(0,1)}^2 \leq \kappa e^{-\frac{\mu}{2} t} \norm{(\bar u_0, \bar u_1)}^2_{H^1(0,1)\times L^2(0,1)},
\end{equation*}
whenever 
$\frac{\abs{bd}}{\abs{c}} \leq \sqrt{\frac{\mu e^{-\mu}}{\sinh{\mu}+2\cosh{\mu}}}$ and $0<\mu \leq \mu^\ast$.
\end{proposition}
\begin{proof}
Taking the time derivative of $V_1$ along the strong solutions of \eqref{reduced}, we get
\begin{equation*}
\begin{aligned} \frac{d}{dt} V_1(\bar u, \bar u_t) &= \int_0^1 2 e^{\mu x}(\bar u_x+\bar u_t)(\bar u_{xt}+\bar u_{tt})+2e^{-\mu x}(\bar u_x-\bar u_t)(\bar u_{xt}-\bar u_{tt})dx.
\end{aligned}
\end{equation*}
Integrating by parts, it follows that
\begin{equation*}
\begin{aligned}
\frac{d}{dt}V_1(\bar u, \bar u_t)&=-\mu V_1(\bar u, \bar u_t) +\left(e^{\mu x}(\bar u_x+\bar u_t)^2 - e^{-\mu}(\bar u_x-\bar u_t)^2\right)\big\vert_{x=0}^{x=1} \\&= -\mu V_1(\bar u, \bar u_t) + \left(e^\mu(\bar u_x(1,t)+\bar u_t(1,t))^2- e^{-\mu}(\bar u_x(1,t)-\bar u_t(1,t))^2 \right).
\end{aligned}
\end{equation*}
Using boundary conditions we get 
\begin{equation*}
\begin{aligned}
\frac{d}{dt} V_1(\bar u, \bar u_t)&=-\mu V_1(\bar u, \bar u_t)+ \left((1- a)^2e^\mu - (1+a)^2 e^{-\mu}\right)\abs{\bar u_t(1,t)}^2 +
2\sinh(\mu) \left(\tfrac{bd}{c}\right)^2\abs{\bar u(1,t)}^2,\\&+e^\mu 2(1-a)\bar u_t(1,t)\tfrac{bd}{c} \bar u(1,t) +  e^{-\mu}2(1+a)\bar u_t(1,t)\tfrac{bd}{c}\bar{u}(1,t).
\end{aligned}
\end{equation*}
Using Young's inequality in the crossed terms, it holds that
\begin{equation}
\label{young_eta}
\begin{aligned}
\frac{d}{dt} V_1(\bar u, \bar u_t)&=-\mu V_1(\bar u, \bar u_t)+ \left((1- a)^2e^\mu - (1+a)^2 e^{-\mu}\right)\abs{\bar u_t(1,t)}^2 +
2\sinh(\mu) \left(\tfrac{bd}{c}\right)^2\abs{\bar u(1,t)}^2,\\&+ \delta \left( e^{\mu}(1-a)^2 +e^{-\mu}(1+a)^2\right) \abs{\bar u_t(1,t)}^2 + \frac{2}{\delta} \left(\frac{bd}{c}\right)^2\cosh{(\mu)} \abs{\bar u(1,t)}^2.
\end{aligned}
\end{equation}
Taking $\delta=\frac12$,  we obtain that
\begin{equation}
\label{lyapunov_reduced}
\begin{aligned} \frac{d}{dt} V_1(\bar u, \bar u_t) &\leq -\mu V_1(\bar u, \bar u_t) -q(a,\mu)\abs{\bar u_t(1,t)}^2 + (2\sinh(\mu)+4\cosh{\mu}) \left(\tfrac{bd}{c}\right)^2\abs{\bar u(1,t)}^2,\end{aligned}
\end{equation}
where $q(a,\mu)$ is given in Remark \ref{qmu}.  Using Lemma \ref{trace_reduced}, and by asking $a,b,c,d$ to satisfy $(i)-(v)$ of Theorem \ref{claim_stab} we have that $q(a,\mu)\geq 0$. As a consequence, we conclude the desired result. Therefore, we get the exponential stability of the reduced system for initial data $(\bar u_0,\bar u_1)\in H^2(0,1) \times H^1(0,1)$ such that $\bar u_0 (0) =0$. Since $(\bar u_0,\bar u_1)\in H^2(0,1) \times H^1(0,1)$ such that $\bar u_0 (0) =0$ is dense in $H^1(0,1)\times L^2(0,1)$ with $u_0(0)=0$, we have that the exponential decay holds for less regular initial data, i.e., for any $(\bar u_0, \bar u_1)\in H^1(0,1)\times L^2(0,1)$ the unique mild solution $(\bar u, \bar u_t)\in C([0,\infty);H^1(0,1)\times L^2(0,1))$ to \eqref{reduced} is exponentially stable around the origin in the $H^1(0,1)\times L^2(0,1)-$norm.
\end{proof}
\begin{remarkx}
\label{a-eta-remark}
Taking $\delta$ such that $\delta\in (0,1)$ in \eqref{young_eta} also works. We have to ensure that the term involving the trace $\abs{\bar u_t(1,t)}^2$ has negative sign. In order to ease the notation in this article we take $\delta=\frac12$. Taking $\delta$ closer to $1$ allows us more flexibility in the parameter $a$, since $\eta>1$ is such that
$$ (1+\delta) e^\mu(1-a)^2-(1-\delta) e^{-\mu}(1+a)^2 \leq 0, \hspace{.1cm} \text{ for } \hspace{.1cm}  a\in [\eta^{-1},\eta], \hspace{.2cm} a\not=1. $$
\end{remarkx}
Following \cite[Section III]{cerpa2019singular}, and as it will be justified later, we are going to treat an extra source term that depends on the trace $\bar u_t(1,\cdot)$. We end this section with the following corollary that concerns the trace $\bar u_t(1,\cdot)$ through the following observability inequality.
\begin{corollary}
\label{r_coro}
Let $a\in \R$ satisfying $(i)$ in Theorem \ref{claim_stab}. Then, the unique solution $\bar u$ of \eqref{reduced} satisfies
$$ q(a,\mu)\int_0^t e^{-\frac{\mu}{2}(t-s)}\abs{\bar u_t(1,s)}^2 ds \leq \kappa e^{-\frac{\mu}{2}t}\norm{(\bar u_0,\bar u_1)}^2_{H^1(0,1)\times L^2(0,1)},$$
whenever 
$\frac{\abs{bd}}{\abs{c}} \leq \sqrt{\frac{\mu e^{-\mu}}{\sinh{\mu}+2\cosh{\mu}}}.$ Here $\kappa=\kappa(\mu)$ is the constant stated in Proposition \ref{reduced_stab}.
\end{corollary}
This corollary follows by using Lemma \ref{trace_reduced} to absorb the term $\abs{\bar u(1,t)}^2$. Then, we use Gronwall's Lemma in the resulting differential inequality. This result also implies that, it generates a suitable admissible observation operator in the sense of \cite[Definition 4.3.1]{tucsnak2009observation}.

% CORRECTED UNTIL HERE

\subsection{Boundary-layer system} 
\label{BoundaryStab}
Let us define $\tau:=\frac{t}{\varepsilon}$ and consider $\bar p(x,\tau):= p(x,\tau) - \left(\frac{d}{c}  + x d\right) u(1,\tau)$, defined for all $(x,\tau)\in [0,1]\times [0,\infty)$. Note that,
\begin{equation*}
\bar p_\tau(x,\tau) = p_\tau(x,\tau) - \varepsilon  \left(\frac{d}{c}  + x d\right) u_t(1,\tau), \hspace{.3cm} \forall (x,\tau)\in  [0,1]\times (0,\infty).
\end{equation*}
Therefore, taking $\varepsilon=0$ yields to $\bar p_\tau(x,\tau) = p_\tau(x,\tau)$. Similarly, one has $\bar p_{xx}(x,\tau)=p_{xx}(x,\tau)$. For the boundary conditions, one needs to compute the space derivative of $\bar p$, that is defined as:
$$
\bar p_x(x,\tau):= p_x(x,\tau) - d u(1,\tau),\: \forall (x,\tau)\in (0,1)\times (0,\infty).
$$
The boundary condition at $x=0$ is given by
$\bar{p}_x(0,\tau) = p_x(0,\tau)-du(1,\tau)=cp(0,\tau)-p_x(1,\tau),$
and thus
$c\bar{p}(0,\tau)=p_{x}(0,\tau) - du(1,\tau)=\bar{p}_x(0,\tau).$ 
For the boundary condition at $x=1$, we have directly 
$\bar p_x(1,\tau)=p_x(1,\tau)-du(1,\tau) =0.$
Finally, the boundary-layer system is given by
\begin{equation}
\label{eq:boundary_layer}
\begin{cases}
    \bar p_\tau  = \bar p_{xx}, & (x,\tau)\in (0,1)\times (0,\infty), \\
    \bar p_x(0,\tau) = c \bar p(0,\tau), & \tau\in(0,\infty),\\
    \bar p_x(1,\tau) = 0,& \tau\in(0,\infty),\\
    \bar p(x,0) = \bar p_0(x),& x\in(0,1).
\end{cases}
\end{equation}
This system has a unique strong solution $\bar p\in C([0,\infty);H^2(0,1))\cap C^1([0,\infty);L^2(0,1))$ provided that $\bar p_0\in H^2(0,1)$ is such that $\bar (p_0)_x(0) =c \bar p_0(0)$ and $\bar (p_0)_x(1)=0$. Indeed,
$$A:D(A)\subset L^2(0,1)\to L^2(0,1),$$
is a closed, densely defined operator with $D(A)=\{p\in H^2(0,1) \text{ : } p_x(0)=cp(0),p_x(1)=0\}$. Simple computations can show that this operator is dissipative. The argument for the maximality of $A$ is the same as the one given in Subsection \ref{well-posed}.
Concerning the exponential stability of \eqref{eq:boundary_layer}, we will use the following Lyapunov functionals
\begin{equation}
\label{Wlyapunov_layer}
W_2(\bar p) = \frac12\norm{\bar p(\cdot,\tau)}_{L^2(0,1)}^2, \hspace{.3cm} \tau \geq 0,
\end{equation} 
\begin{equation}
\label{Vlyapunov_layer}
V_2(\bar p)=\frac{1}{2}\int_0^1 \abs{\bar p_x(\cdot,\tau)}^2dx + \frac{c}{2} \abs{\bar p(0,\tau)}^2,\hspace{.3cm} \tau \geq 0, \hspace{.3cm} c>0  .
\end{equation}
The next result states the exponential decay result for \eqref{eq:boundary_layer} in the $L^2(0,1)-$norm.
\begin{lemma}
For each $c\geq \frac{\pi^2}{4}$ and $\bar p_0\in L^2(0,1)$, the unique mild solution $p\in C([0,\infty);L^2(0,1))$ to \eqref{eq:boundary_layer} satisfies
$$ \norm{\bar p(\cdot,\tau)}^2_{L^2(0,1)} \leq  e^{-\frac{\pi^2}{2}\tau} \norm{\bar p_0}_{L^2(0,1)}, \hspace{.3cm} \forall \tau \geq 0. $$
\end{lemma}
\begin{proof}
First we take $\bar p_0\in H^2(0,1)$ is such that $\bar (p_0)_x(0) =c \bar p_0(0)$ and $\bar (p_0)_x(1)=0$. 
Supposing that $c$ is any positive constant, taking the time derivative of $W_2(\bar p)$ along strong solutions to \eqref{eq:boundary_layer}, we get
\begin{equation}
\begin{aligned}
\frac{d}{d\tau} W_2(\bar p)& =   \int_0^1 \bar p \bar p_\tau dx = \int_0^1 \bar p \bar p_{xx} dx = -\int_0^1 \abs{\bar p_x}^2 dx -c\abs{\bar p(0,\tau)}^2.
\end{aligned} 
\end{equation} 
Using Wirtinger's type we get
\begin{equation}
\label{gronwall_w2}
\frac{d}{d\tau} W_2(\bar p) \leq -\frac{\pi^2}{4} \int_0^1 \abs{\bar p}^2 dx+ \left(\frac{\pi^2}{4}-c\right)\abs{\bar p(0,\tau)}^2.
\end{equation}
It follows directly that if $c\geq \tfrac{\pi^2}{4}$ then the exponential stability in the $L^2(0,1)-$norm is achieved. By density of the set $\{p\in H^2(0,1) \text{ : } p_x(1)=0, \hspace{.3cm} p_x(0)=cp(0)\}$ in $L^2(0,1)$ we have that this result also holds for mild solutions $\bar p\in C([0,\infty);L^2(0,1))$, when taking initial data $\bar p_0\in L^2(0,1)$.
\end{proof}
Now we state the exponential decay of \eqref{eq:boundary_layer} in the $H^1(0,1)-$ norm.
\begin{lemma}
\label{decay_h1}
Let $c\geq \frac{\pi^2}{8}$. Let $\bar p$ be the unique mild solution of \eqref{eq:boundary_layer} with initial condition $p_0\in H^1(0,1)$. There exists a constant $K>0$ depending only on $c>0$ such that
$$\norm{\bar p(\cdot,\tau)}^2_{H^1(0,1)} \leq  K e^{-\frac{\pi^2}{4}\tau} \norm{\bar p_0}^2_{H^1(0,1)}, \hspace{.3cm} \forall \tau \geq 0 .  $$
Moreover, the unique mild solution has the following improved regularity $p\in C([0,\infty);H^1(0,1))$, whenever $p_0\in H^1(0,1)$. The improved regularity is in the sense that the solution is more regular whenever the initial condition is more regular itself.
\end{lemma}
\begin{proof}
Supposing $c>0$, taking the time derivative of \eqref{Vlyapunov_layer} along the strong solution $\bar p$ of \eqref{eq:boundary_layer}, integrating by parts and using boundary conditions we get
\begin{equation*}
     \frac{d}{dt}V_2(\bar p) = \int_0^1 \bar p_{x} \bar p_{xt} dx + c\bar p(0,t)p_t(0,t)= -\int_0^1 \abs{\bar p_{xx}} dx.
\end{equation*}
Using Poincaré's inequality we arrive at $\frac{d}{dt}V_2(\bar p) \leq -\tfrac14\min\{4c,\pi^2\} V_2(\bar p).$ Thus, for each $c>0$ we get the exponential stability in the $H^1(0,1)-$norm.
\end{proof}
Following \cite[Section III]{cerpa2019singular} we are going to treat an extra boundary term that depends on the trace $\abs{\bar p(0,\cdot)}^2$. To conclude this subsection, we state the following result where the trace term $\abs{\bar p(0,\cdot)}^2$ is controlled. This enhance the usefulness of considering both exponential stability results for the boundary-layer system \eqref{eq:boundary_layer} as both are going to play a key role in the proof of Theorem \ref{thikonov}.
\begin{corollary}
\label{bl_coro}
For each $c\geq \frac{\pi^2}{4}$, the unique mild solution $\bar p\in C([0,\infty);L^2(0,1))$ of \eqref{eq:boundary_layer} satisfies
$$ \abs{\bar p(0,\tau)}^2 \leq c^{-1} e^{-\frac{\pi^2}{4}\tau} \norm{\bar p_0}^2_{H^1(0,1)}, \hspace{.3cm} \forall \tau \geq 0, $$
provided that $\bar p_0\in H^1(0,1).$ 
\end{corollary}
\begin{proof}
This Corollary follows directly of the proof of Lemma \ref{decay_h1} and using the improved regularity $\bar p\in C([0,\infty);H^1(0,1))$ whenever $\bar p_0\in H^1(0,1)$
\end{proof}
\subsection{Full system}
\label{FullStab}
In this section we are going to prove the exponential decay of the solutions of \eqref{system} in the $H^1(0,1)\times L^2(0,1)\times L^2(0,1)-$norm and $H^1(0,1)\times L^2(0,1)\times H^1(0,1)-$norm, by using the Lyapunov functionals designed earlier this section. For this purpose, let us consider $\tilde p(t,x) = p(t,x) - \left(\frac{d}{c} + xd\right)u(1,t)$, which is defined for $[0,1]\times \mathbb [0,\infty)$. Then, we have that $$\tilde p_t(t,x) = p_t(t,x) - \left(\frac{d}{c} + xd\right)u_t(1,t), \hspace{.3cm} (x,t)\in  (0,1)\times \mathbb (0,\infty).$$  
One can also prove that $\tilde p_{xx}(t,x) = p_{xx}(t,x)$, for all $(x,t)\in  (0,1)\times \mathbb (0,\infty)$, and compute the boundary conditions. Therefore, $(u,\tilde p)$ satisfies the following coupled PDE system

\begin{equation}
\label{eq:full-system}
\begin{cases}
   \varepsilon \tilde p_t =  \tilde p_{xx} + \varepsilon \left(\frac{d}{c} + xd\right)u_t(1,t),& (x,t)\in (0,1)\times(0,\infty) \\
     \tilde p_x(0,t) = c \tilde p(0,t),& t\in (0,\infty),\\
     \tilde p_x(1,t) =0,& t\in (0,\infty)\\
     \tilde p(x,0) = \tilde p_0(x),& x\in (0,1),\\
     u_{tt}=u_{xx},&(x,t)\in (0,1)\times(0,\infty),\\u(0,t)=0,& t\in (0,\infty), \\ u_x(1,t)=-au_t(1,t)+\tfrac{bd}{c}u(1,t)+b\tilde{p}(0,t),& t\in (0,\infty),\\ u(x,0) = u_0(x),\hspace{.1cm} u_t(x,0) = u_1(x),& x\in (0,1).
\end{cases}
\end{equation}
Now we want to analyze the exponential stability of \eqref{eq:full-system}. For this purpose we propose the Lyapunov functionals   $V(\tilde p, u)=V_1(u)+\varepsilon V_2(\tilde p)$ and $W(\tilde p,u)=V_1(u)+ \varepsilon W_2(\tilde p)$, where $V_1,V_2,W_2$ were defined in \eqref{v1},\eqref{Vlyapunov_layer} and \eqref{Wlyapunov_layer}, respectively. Before we prove Theorem \ref{full_stab}, we state the following technical lemma
\begin{lemma}
Let $a$ satisfies $(i)$ of Theorem \ref{claim_stab}. Then, the unique solution $(u,u_t,\tilde p)$ to \eqref{eq:full-system} satisfies
\begin{equation}
\label{v1Lyapunov}
\begin{aligned}\frac{d}{dt}V_1(u,u_t) \leq -\frac{\mu}{2} V_1(u,u_t) -q(a,\mu)\abs{u_t(1,t)}^2 &+ \left(F(\mu)\left(\tfrac{bd}{c}\right)^2-\mu e^{-\mu}\right)\abs{u(1,t)}^2 + F(\mu)b^2\abs{\tilde p(0,t)}^2,\end{aligned}
\end{equation}
where $F(\mu)$ is defined in Theorem \ref{full_stab} and $q(a,\mu)$ is defined in Remark \ref{qmu}.
\end{lemma}
\begin{proof}
Taking the time derivative along the strong solutions of \eqref{eq:full-system} we have that
\begin{multline*}
\frac{d}{dt} V_1(u,u_t) = -\mu V_1(u,u_t)  + \left((1-a)^2e^\mu - (1+a)^2 e^{-\mu}\right)\abs{u_t(1,t)}^2 + 2\sinh{\mu}\left(\tfrac{bd}{c}\right)^2\abs{u(t,1)}^2,\\ + 2\sinh(\mu)\abs{b\tilde{p}(0,t)}^2
+e^\mu 2\left(\tfrac{db}{c}\right)u(1,t)b\tilde{p}(0,t)  +e^\mu 2(1-a)u_t(1,t)b\tilde{p}(0,t) +e^\mu 2(1-a)u_t(1,t)\left(\tfrac{db}{c}\right)u(1,t),\\+e^{-\mu} 2\left(\tfrac{db}{c}\right)u(1,t)b\tilde{p}(0,t)  -e^{-\mu} 2(1+a)u_t(1,t)b\tilde{p}(0,t) -e^{-\mu} 2(1+a)u_t(1,t)\left(\tfrac{db}{c}\right)u(1,t),    
\end{multline*}
by using Young and Cauchy-Schwarz inequalities, we get
\begin{multline*}
\frac{d}{dt} V_1(u,u_t) = -\mu V_1(u,u_t) + (2\sinh{(\mu)}+ 2\cosh{(\mu)} +\tfrac{2}{\delta}\cosh{(\mu)} ) b^2\abs{\tilde p(0,t)}^2,\\+ 2\delta \left(e^{\mu}(1-a)^2 +e^{-\mu}(1+a)^2\right)+ \abs{u_t(1,t)}^2 (2\sinh{(\mu)}+ 2\cosh{(\mu)} +\tfrac{2}{\delta}\cosh{(\mu)} ) \left(\frac{bd}{c}\right)^2\abs{u(1,t)}^2, \\ +\left((1-a)^2e^\mu - (1+a)^2 e^{-\mu}\right)\abs{u_t(1,t)}^2.
\end{multline*}
Taking $\delta=\frac{1}{4}$, and using Lemma \ref{trace_reduced} for $\abs{u(1,t)}^2$ we get the desired result. Note that Remark \ref{a-eta-remark} also applies in this case.
\end{proof}
Now we are ready to prove Theorem \ref{full_stab}.

\begin{proof}[\textbf{Proof of Theorem \ref{full_stab}}]  

\begin{itemize}
\item[(i)] For any strong solution of \eqref{eq:full-system} and taking the derivative of \eqref{Wlyapunov_layer}, we have the following
\begin{equation*}
\begin{aligned} \varepsilon \frac{d}{dt} W_2(\tilde p) = -\int_0^1 \abs{\tilde{p}_x}^2 - c \abs{\tilde{p}(0,t)}^2+ \int_0^1\varepsilon\tilde{p} \left(\frac{d}{c} +xd\right)u_{t}(1,t) dx.\end{aligned}
\end{equation*}
Using Wirtinger's and Young's inequalities we get
\begin{equation*}
\begin{aligned} \varepsilon \frac{d}{dt} W_2(\tilde p) &\leq -\frac{\pi^2}{4}\int_0^1 \abs{\tilde p}^2 dx + \left(\frac{\pi^2}{4}-c\right)\abs{\tilde p(0,t)}^2 + \frac{\varepsilon^2 d^2(1+c)^2}{2c^2\delta}\int_0^1 \abs{\tilde p}^2 dx + \frac{\delta}{2}\abs{u_t(1,t)}^2, \end{aligned}
\end{equation*}
thus
\begin{multline}
\label{w2}
\varepsilon \frac{d}{dt} W_2(\tilde p) \leq \left(\frac{\varepsilon^2 d^2(1+c)^2}{2c^2\delta}-\frac{\pi^2}{8}\right)\int_0^1 \abs{\tilde p}^2 dx + \left(\frac{\pi^2}{4}-c\right)\abs{\tilde p(0,t)}^2  dx + \frac{\delta}{2}\abs{u_t(1,t)}^2dx, \\ -\frac{\pi}{8}\int_0^1 \abs{\tilde p} dx.
\end{multline}
Using inequalities \eqref{v1Lyapunov} and \eqref{w2} we get
\begin{multline*}
\frac{d}{dt} W(u,u_t,\tilde p) \leq -\frac{\mu}{2} V_1(u,u_t) -\frac{\pi^2}{4\varepsilon} \varepsilon W_2(\tilde p) + \left(\frac{\delta}{2} -q(a,\mu)\right) \abs{u_t(1,t)}^2,\\+\left(\frac{\pi^2}{4} + F(\mu)b^2 -c\right) \abs{\tilde p(0,t)}^2 + \left(F(\mu)\left(\tfrac{bd}{c}\right)^2-\mu e^{-\mu}\right)\abs{u(1,t)}^2,\\ + \left(\frac{\varepsilon^2 d^2(1+c)^2}{2c^2\delta}-\frac{\pi^2}{8}\right)\int_0^1\abs{\tilde p}^2 dx.    
\end{multline*}
Defining $\varepsilon_1= \frac{\sqrt{q(a,\mu)}\pi c}{ 2\abs{d} (1+c)}$ and $\varepsilon_2 = \frac{\pi^2}{2\mu}$. For $\mu <\mu^\ast$, where $\mu^\ast$ is given in Theorem \eqref{full_stab}, we have that $q(a,\mu)>0$. Taking  $\delta=2q(a,\mu)$, $\varepsilon\in (0,\varepsilon_1]$, and $a,b,c,d$ satisfying $(i)-(v)$ of Theorem \eqref{full_stab} we arrive at
\begin{equation*}
\frac{d}{dt} W(u,u_t,\tilde p) \leq -\min\left\{\frac{\mu}{2},\frac{\pi^2}{4\varepsilon}\right\} W(u,u_t,\tilde p) \leq -\frac{\mu}{2}W(u,u_t,\tilde p),
\end{equation*}
where the last inequality holds for $\varepsilon\in (0,\varepsilon_2)$. Finally, taking $\varepsilon^\ast_1 = \min\{\varepsilon_1,\varepsilon_2\}$ and using some suitable density argument, we get the desired conclusion.
\item[(ii)] For any strong solution of \eqref{eq:full-system} and taking the derivative of \eqref{Vlyapunov_layer} we have that
\begin{equation*}
 \frac{d}{dt} \left(\varepsilon V_2(\tilde p)\right) =- \int_0^1 \abs{\tilde{p}_{xx}}^2 dx+\int_0^1 \varepsilon \left(\frac{d}{c}+dx\right)u_{t}(1,t)\tilde{p}_{xx}dx, 
\end{equation*}
using Young's inequality we have that
\begin{multline}
\label{v2}
\varepsilon \dot{V}_2(\tilde{p}) \leq -\int_0^1\abs{\tilde p_{xx}}^2dx + \int_0^1\frac{\varepsilon(1+c)\abs{d}}{c}\abs{\tilde p_{xx}}\abs{u_t(1,t)}dx \leq \left(\frac{\varepsilon^2(1+c)^2d^2}{2\delta c^2}-\frac{1}{3}\right)\int_0^1\abs{\tilde p_{xx}}^2dx,\\ - \frac{2}{3}\int_0^1\abs{\tilde p_{xx}}^2dx+ \frac{\delta}{2}\abs{u_t(1,t)}^2.
\end{multline}
Combining \eqref{v1Lyapunov} and \eqref{v2} we see that
\begin{multline}
\label{combined}
\frac{d}{dt} V(u,u_t,\tilde p) \leq -\frac{\mu}{2} V_1(u,u_t)-\frac{2}{3}\int_0^1 \abs{\tilde p_{xx}}^2 + \left(\frac{\delta}{2}-q(a,\mu)\right)\abs{u_t(1,t)}^2 + F(\mu)b^2 \abs{\tilde p(0,t)}^2,\\+ \left(F(\mu)\left(\tfrac{bd}{c}\right)^2-\mu e^{-\mu}\right)\abs{u(1,t)}^2+ \left(\frac{\varepsilon^2(1+c)^2d^2}{2\delta c^2}-\frac{1}{3}\right)\int_0^1 \abs{\tilde p_{xx}}^2.
\end{multline}
Applying Lemma \ref{trace_poincare} for $p_x(0,t)$, and using boundary conditions we have $$c^2 \abs{\tilde p(0,t)}^2 \leq \int_0^1 \abs{\tilde p_{xx}}^2dx.$$
Since $c\geq \tfrac{\pi^2}{8}>1$, we get the following inequality
\begin{equation}
\label{h2_trace}
-c\int_0^1 \abs{\tilde p_{xx}}^2 dx \leq -\int_0^1 \abs{\tilde p_{xx}}^2 \leq -c^2 \abs{\tilde p(0,t)}^2.
\end{equation}
By using \eqref{h2_trace} in \eqref{combined} we get
\begin{multline*}
\frac{d}{dt} V(u,u_t,\tilde p) \leq -\frac{\mu}{2} V_1(u,u_t)-\frac{1}{3}\int_0^1 \abs{\tilde p_{xx}}^2 + \left(\frac{\delta}{2}-q(a,\mu)\right)\abs{u_t(1,t)}^2,\\ + \left(F(\mu)b^2-\frac{c}{3}\right) \abs{\tilde p(0,t)}^2 + \left(F(\mu)\left(\tfrac{bd}{c}\right)^2-\mu e^{-\mu}\right)\abs{u(1,t)}^2+ \left(\frac{\varepsilon^2(1+c)^2d^2}{2\delta c^2}-\frac{1}{3}\right)\int_0^1 \abs{\tilde p_{xx}}^2.
\end{multline*}
We have the following estimate thanks to Poincaré's Inequality $-\int_0^1 \abs{\tilde p_{xx}}^2 dx \leq - \frac{2c\pi^2}{4c+\pi^2} V_2(\tilde p)$, so the last inequality is equivalent to
\begin{multline*}
\frac{d}{dt} V(u,u_t,\tilde p) \leq -\frac{\mu}{2} V_1(u,u_t)-\frac{2c\pi^2}{3(4c+\pi^2)}\frac{1}{\varepsilon} \left(\varepsilon V_2(\tilde p)\right)+ \left(\frac{\delta}{2}-q(a,\mu)\right)\abs{u_t(1,t)}^2,\\ + \left(F(\mu)b^2-\frac{c}{3}\right) \abs{\tilde p(0,t)}^2+ \left(F(\mu)\left(\tfrac{bd}{c}\right)^2-\mu e^{-\mu}\right)\abs{u(1,t)}^2,\\+ \left(\frac{\varepsilon^2(1+c)^2d^2}{2\delta c^2}-\frac{1}{3}\right)\int_0^1 \abs{\tilde p_{xx}}^2.
\end{multline*}
Defining $\varepsilon_1= \tfrac{2c\sqrt{q(a,\mu)}}{(1+c)\abs{d}}$, and $\varepsilon_2=\frac{4c\pi^2}{\mu(4c+\pi^2)}$. For $\mu <\mu^\ast$, where $\mu^\ast$ is given in Theorem \eqref{full_stab}, we have that $q(a,\mu)>0$.. Taking $\delta=2q(a,\mu)$, and the parameters $a,b,d,c,\mu$ as in the hypothesis we have that for $\varepsilon\leq \varepsilon_1$
\begin{equation*}
\frac{d}{dt} V(u,u_t,\tilde p) \leq -\frac{\mu}{2} V_1(u,u_t)- \frac{2c\pi^2}{4c+\pi^2}\frac{1}{\varepsilon}  \left(\varepsilon V_2(\tilde p)\right) \leq -\min\{\tfrac{\mu}{2}, \tfrac{2c\pi^2}{4c+\pi^2}\tfrac{1}{\varepsilon}\} V(u,u_t,\tilde p),
\end{equation*}
where the last inequality holds for $\varepsilon\in (0,\varepsilon_2)$. Finally, taking $\varepsilon^\ast_2 = \min\{\varepsilon_1,\varepsilon_2\}$ we get the desired result for strong solutions. On the other hand,  following Remark \ref{char} it follows that $D(A_\varepsilon)$ is dense in $H^1(0,1)\times L^2(0,1)\times H^1(0,1)$. Therefore, we can extend the semigruoup associted to \eqref{system} by continuity to the space $H^1(0,1)\times L^2(0,1)\times H^1(0,1)$. In consequene, this unique mild solution is also exponentially stable.
\end{itemize} \end{proof}

\section{Tikhonov Theorem}
\label{Thiko}
At the moment we have shown that, the exponential stability of the approximated subsystems given by the SPM, we deduce the stability of \eqref{system} in the $H^1(0,1)\times L^2(0,1)\times L^2(0,1)-$norm and the $H^1(0,1)\times L^2(0,1)\times H^1(0,1)-$norm, for $\varepsilon>0$ small enough. In this section, we are going to prove Theorem \ref{thikonov}, which uses both subsystems to give an approximation of the full system. This Tikhonov result is for $H^1(0,1)\times L^2(0,1)\times L^2(0,1)-$norm stability analysis.  The idea is to show how the solution of \eqref{system} goes to zero while explaining how it does when the parameter $\varepsilon$ is small enough. Being precise, we define the so called error system, which is the result of approximating each variables $(u,p)$ of \eqref{system} by the solutions $\bar u$,$\bar p$ of subsystems \eqref{eq:reduced},\eqref{boundary_layer}, respectively and the quasi steady-state given in Section \ref{Stab}.

Let us introduce the following variables
\begin{equation*}
\alpha(x,t)=u(x,t)-\bar{u}(x,t), \hspace{.3cm} \forall (x,t)\in [0,1]\times [0,\infty),
\end{equation*}
and
\begin{equation*}
\beta(x,t)= p(x,t)-\left(\tfrac{d}{c}+xd\right)\bar u(1,t)-\bar{p}(x,\tfrac{t}{\varepsilon}), \hspace{.3cm} \forall (t,x)\in [0,1]\times [0,\infty),
\end{equation*}
where $(u,p),\bar u, \bar p $ are the strong solutions to \eqref{system}, \eqref{eq:reduced} and \eqref{boundary_layer}, respectively.

Computing successively for $\alpha,\beta$ we have
\begin{equation*}
\begin{aligned}
&\alpha_{t}(x,t) = u_{t}(x,t)-\bar u_{t}(x,t),\hspace{.3cm}
\alpha_{tt}(x,t) = u_{tt}(x,t)-\bar u_{tt}(x,t), \hspace{.3cm} \forall (t,x)\in (0,1)\times (0,\infty),\\&
\alpha_{x}(x,t) = u_{x}(x,t)-\bar u_{x}(x,t), \hspace{.3cm}
\alpha_{xx}(x,t)= u_{xx}(x,t)-\bar u_{xx}(x,t), \hspace{.3cm} \forall (t,x)\in (0,1)\times (0,\infty),
\end{aligned}
\end{equation*}
and
\begin{equation*}
\begin{aligned}
&\beta_{t}(x,t) = p_{t}(x,t)-\left(\tfrac{d}{c}+xd\right)\bar u_{t}(1,t)-\frac{1}{\varepsilon} \bar p_\tau (x,\tfrac{t}{\varepsilon}), \hspace{.3cm} \forall (t,x)\in (0,1)\times (0,\infty),\\&
\beta_{x}(x,t) = p_{x}(x,t)- d\bar u(1,t)-\bar p_x(x,\tfrac{t}{\varepsilon}), \hspace{.3cm} \forall (t,x)\in (0,1)\times (0,\infty),\\&
\beta_{xx}(x,t) = p_{xx}(x,t)-\bar p_{xx}(x,t),\hspace{.3cm} \forall (t,x)\in (0,1)\times (0,\infty).
\end{aligned}
\end{equation*}
Thus, $\alpha,\beta$ satisfies the following error system
\begin{equation*}
\begin{cases}
\alpha_{tt} = \alpha_{xx}, & (x,t)\in (0,1)\times (0,\infty),  \\
\varepsilon\beta_{t}  = \beta_{xx}-\varepsilon\left(\frac{d}{c}+xd\right)\bar{u}_t(1,t), & (x,t)\in (0,1)\times (0,\infty), \\
\end{cases}
\end{equation*}
with boundary conditions
\begin{equation*}
\begin{cases}
\alpha(0,t) = 0, & t\in (0,\infty), \\
\alpha_x(1,t) = - a \alpha_t(1,t) +b\beta(0,t)+\bar p(\tfrac{t}{\varepsilon},0), & t\in (0,\infty), \\
\beta_x(0,t)  = c\beta(0,t), & t\in (0,\infty), \\
\beta_x(1,t) = d\alpha(1,t), & t\in (0,\infty),
\end{cases}
\end{equation*}
and initial condition
\begin{equation*}
\begin{cases}
\alpha (x,0) = u_0(x)-\bar u_0(x), & x\in (0,1),\\ \alpha_t(x,0) = u_1(x)-\bar u_1(x), & x\in (0,1), \\
\beta (x,0) = p_0(x)-\bar p_0(x) - \left(\tfrac{d}{c}+xd\right)\bar u_0(1), & x\in (0,1).
\end{cases}
\end{equation*}
We define the following Lyapunov functional $\tilde W(\alpha,\alpha_t,\beta)= V_1(\alpha,\alpha_t)+\varepsilon W_2(\beta)$.

We start with the two following lemmas.
\begin{proposition}
\label{equivalence}
The Lyapunov functional $\tilde W(\alpha,\beta)$ satisfies
\begin{equation*}
 O(\varepsilon) \norm{\beta}^2_{L^2(0,1)} \leq \tilde W(\alpha,\alpha_t,\beta), \hspace{.5cm} O(1) \left( \norm{\alpha}^2_{H^1(0,1)}+\norm{\alpha_t}^2_{L^2(0,1)}\right)\leq  \tilde W(\alpha,\alpha_t,\beta),
\end{equation*}
and
\begin{equation*}
\tilde W(\alpha,\alpha_t,\beta) \leq O(1) \left( \norm{\alpha}^2_{H^1(0,1)}+\norm{\alpha_t}^2_{L^2(0,1)}+ \norm{\beta}^2_{L^2(0,1)}\right).
\end{equation*}
\end{proposition}
\begin{proof}
This is a straightforward result from the definition of $\tilde W$ and the equivalence between $V_1(\alpha)$ and the $H^1(0,1)-$norm of $\alpha$. 
\end{proof}
\begin{lemma}
For every $\mu>0$ and $a>0$ we have that
$$ \frac{d}{dt} V_1(\alpha,\alpha_t)= -\mu V_1(\alpha,\alpha_t) + F(\mu)\bar p(\tfrac{t}{\varepsilon},0) + F(\mu)b^2 \abs{\beta(0,t)}^2 - q(a,\mu)\abs{\alpha_t(1,t)}^2.$$
\end{lemma}
\begin{proof}
It is a consequence of computing the derivative of $V_1$ and using Young's inequality as it was used in the proof of Lemma \ref{v1Lyapunov}.
\end{proof}
\begin{lemma}
For every $c>0$ and $d\in \R$ we have that
\begin{multline}
\frac{d}{dt} \varepsilon W_2(\beta) \leq  - \frac{\pi^2}{8} \int_0^1 \abs{\beta}^2dx +\frac{1}{2} \abs{\bar u_t(t,1)}^2 +\left(\frac{\varepsilon^2 d^2(c+1)^2}{2c^2}-\frac{\pi^2}{8}\right)\int_0^1\abs{\beta}^2 dx,\\ d^2\abs{\alpha(1,t)}^2+\left(\frac{\pi^2+2}{4}-c \right)\abs{\beta(0,t)}^2.
\end{multline}
\end{lemma}
\begin{proof}
Note that
\begin{equation*}
\begin{aligned}
\frac{d}{dt} \varepsilon W_2(\beta) & = \int_0^1 \varepsilon \beta_t \beta dx = \int_0^1 \beta \beta_{xx}+ \int_0^1 \varepsilon(\tfrac{d}{c}+xd)\bar u_t(1,t) \beta(x,t) dx, \\&=-\int_0^1 \abs{\beta_x}^2+ (\beta\beta_x)\big\vert_{x=0}^{x=1}+\int_0^1 \varepsilon(\tfrac{d}{c}+xd)\bar u_t(1,t) \beta(x,t) dx,\\&= -\int_0^1 \abs{\beta_x}^2+ d\alpha(1,t)\beta(1,t)-c\abs{\beta(0,t)}^2+\int_0^1 \varepsilon(\tfrac{d}{c}+xd)\bar u_t(1,t) \beta(x,t) dx.
\end{aligned}
\end{equation*}
Using Young's inequality and Cauchy-Schwarz properly, we get
\begin{multline}
\varepsilon \dot W_2(\beta) \leq - \int_0^1 \abs{\beta_x}^2dx -c\abs{\beta(0,t)}^2 + \frac{d^2}{2\delta}\abs{\alpha(1,t)}^2 + \frac{\delta}{2}\abs{\beta(1,t)},\\ + \frac{\varepsilon^2 d^2 (1+c)^2}{2c^2} \int_0^1 \abs{\beta}^2dx + \frac{1}{2} \abs{\bar u_t(1,t)}^2,
\end{multline}
employing Lemma \ref{wirtinger} together with Lemma \ref{a1}, and taking $\delta=\frac12$, we arrive at
\begin{multline}
\varepsilon \dot W_2(\beta) \leq - \frac{\pi^2}{4} \int_0^1 \abs{\beta}^2dx + \left(\frac{\pi^2}{4} +\frac{1}{2} -c \right) \abs{\beta(0,t)}^2 + d^2\abs{\alpha(1,t)}^2,  \\+ \frac{\varepsilon^2 d^2 (1+c)^2}{2c^2} \int_0^1 \abs{\beta}^2dx + \frac{1}{2} \abs{\bar u_t(1,t)}^2.
\end{multline}
\end{proof}
Now we are able to prove the Tikhonov result stated in Theorem \ref{thikonov}.
\begin{proof}[\textbf {Proof of Theorem \ref{thikonov}:}]
Using the previous lemmas stated in this section we get that
\begin{multline*}
\frac{d}{dt} \tilde W(\alpha,\alpha_t,\beta) \leq  -\frac{\mu}{2} V_1(\alpha,\alpha_t)+ \left(\frac{\varepsilon^2 d^2 (1+c)^2}{c^2}-\frac{\pi^2}{4}\right)\frac{1}{\varepsilon} \varepsilon W_2(\beta) -\frac{\pi^2}{4\varepsilon} \varepsilon W_2(\beta)\\ + \left(d^2-\mu e^{-\mu} \right)\abs{\alpha(1,t)}^2 + \left(\frac{\pi^2+2}{4} +F(\mu) b^2 - c \right) \abs{\beta(0,t)}^2 - q(a,\mu)\abs{\alpha_t(1,t)}^2 \\+ F(\mu)\bar p(0,\tfrac{t}{\varepsilon})+ \frac12 \abs{\bar u_t(1,t)}^2.
\end{multline*}
Defining $\varepsilon_1=\frac{c\pi }{2\abs{d}(1+c)}$
and $\varepsilon_2= \frac{\pi^2}{2\mu}$. Now, taking $\varepsilon\in(0,\varepsilon_1)$ and letting $a,b,c,d$ satisfy $(i)-(v)$ in Theorem \ref{claim_stab} and $b,c,d\in \R$ satisfying $(1),(2),(3)$ of Theorem \ref{thikonov}, we get that
\begin{equation}
\label{gronwall_thiko}
\begin{aligned}
\frac{d}{dt}\tilde{W}(\alpha,\alpha_t,\beta) &\leq -\min\left\{\frac{\mu}{2}, \frac{\pi^2}{4\varepsilon}\right\} \tilde W(\alpha,\alpha_t,\beta)+ F(\mu) \abs{\bar p(\tfrac{t}{\varepsilon},0)}^2 + \frac{1}{2} \abs{\bar u_t(1,t)}^2,\\&\leq -\frac{\mu}{2} \tilde W(\alpha,\alpha_t,\beta)+ F(\mu) \abs{\bar p(\tfrac{t}{\varepsilon},0)}^2 + \frac{1}{2} \abs{\bar u_t(1,t)}^2,
\end{aligned}
\end{equation}
where the last inequality holds for $\varepsilon\in (0,\varepsilon_2)$. Defining $\varepsilon^\ast_3 = \min\left\{\varepsilon_1,\varepsilon_2\right\}$, and taking $\varepsilon\in (0,\varepsilon^\ast_3)$ we get
from Corollary \ref{bl_coro} that
\begin{equation*}
\begin{aligned}
\frac{d}{dt}  \tilde W(\alpha,\alpha_t,\beta) &\leq - \frac{\mu}{2} \tilde W(\alpha,\alpha_t,\beta) + F(\mu)e^{-\frac{\pi^2}{4\varepsilon} t}\norm{\bar p_0}_{H^1(0,1)} + \frac{1}{2}\abs{\bar u_t(1,t)}^2, \\& \leq - \frac{\mu}{2} \tilde W(\alpha,\alpha_t,\beta) + F(\mu)e^{-\frac{\mu}{4} t}\norm{\bar p_0}_{H^1(0,1)} + \frac{1}{2}\abs{\bar u_t(1,t)}^2.
\end{aligned}
\end{equation*}
Multiplying by $e^{\frac{\mu}{2}t}$ and integrating between $0$ and $t$ we deduce
\begin{equation*}
\begin{aligned}
\tilde W(\alpha,\alpha_t,\beta)  \leq e^{\frac{-\mu}{2}t}  \left(W(\alpha(x,0),\alpha_t(x,0),\beta(x,0))+  F(\mu)\frac{\mu}{4}e^{\frac{\mu}{4} t}\norm{\bar p_0}_{H^1(0,1)} +\frac12\int_0^t  e^{\frac{\mu}{2}s} \abs{\bar u_t(1,s)}^2ds\right), 
\end{aligned}
\end{equation*}
and by using Corollary \ref{r_coro}, Proposition \ref{equivalence} and smallness conditions \eqref{small-difference}, \eqref{small-sub}, we arrive at 
\begin{equation}
\label{approximation}
\begin{aligned}
\tilde W(\alpha,\alpha_t,\beta)  &\leq O(1) e^{-\frac{\mu}{4}t}  \left(W(\alpha(x,0),\alpha_t(x,0),\beta(x,0))+\norm{p_0}^2_{H^1(0,1)}+\norm{(\bar u_0,\bar u_1)}^2_{H^1(0,1)\times L^2(0,1)} \right), \\& \leq O(\varepsilon^{3}) e^{-\frac{\mu}{4} t} .
\end{aligned}
\end{equation}
Therefore, we have that the desired result hold for any strong solution $(u,p),\bar u, \bar p$  to \eqref{system}, \eqref{eq:reduced} and \eqref{boundary_layer}, respectively. By using a suitable density argument, we can show that \eqref{approximation} holds for any mild solution $(u,p),\bar u, \bar p$  to \eqref{system}, \eqref{eq:reduced} and \eqref{boundary_layer}, respectively.
\end{proof}
\section{Conclusions and open problems}
In this paper, we have provided a singular perturbation analysis for a parabolic-hyperbolic system
composed by a fast heat equation and a slow wave equation. In particular, we have proved that, the conditions
for the reduced order system and the boundary-layer system to be exponentially system
also work for the full-system for $\varepsilon$ small enough, and the decay is in the $H^1(0,1)\times L^2(0,1) \times L^2(0,1)-$norm. Furthermore, we provide a Tikhonov approximation for the full-system trough the subsystems given by the method. On the other hand, by using a slightly different Lyapunov functional for the heat equation we have shown a better result for the expential stability, namely, we show an exponential decay in the $H^1(0,1)\times L^2(0,1) \times H^1(0,1)-$norm, however we are not able to prove a Tikhonov result for this case. \\
Before concluding this article we would like to talk about an open problem when changing a coupling term in \eqref{system}. Consider the following fast heat - slow wave coupled system 
\begin{equation}
\label{diff_coupling}
\begin{cases}
    u_{tt}  = u_{xx}, & (x,t)\in (0,1)\times (0,\infty),  \\
    u(0,t) = 0, & t\in (0,\infty),\\
    u_x(1,t) = -  au_t(1,t) + b p(0,t), & t\in (0,\infty),\\
    %u(x,0) = u_0(x),\: u_t(x,0) = u_1(x), & x\in (0,1),\\
    \varepsilon p_t  = p_{xx},& (x,t)\in (0,1)\times (0,\infty),\\
    p_x(0,t) = c p(0,t),& t\in (0,\infty),\\
    p_x(1,t) = d u_t(1,t),& t\in (0,\infty),
    %p(x,0) = p_0(x),& x\in (0,1),
\end{cases}
\end{equation}
where we have the coupling $p_x(1,t) = d u_t(1,t)$ instead of $p_x(1,t) = d (1,t)$. For \eqref{diff_coupling} the resulting subsystems can be treated as in Section \ref{Stab}, and therefore they are exponentially stable under reasonable conditions. However when writing the equivalent full-system, like in Section \ref{FullStab}, we end up with
\begin{equation}
\label{full_sys}
\begin{cases}
   \varepsilon \tilde p_t =  \tilde p_{xx} + \varepsilon \left(\frac{d}{c} + xd\right)u_{tt}(1,t),& (x,t)\in (0,1)\times(0,\infty) \\
     \tilde p_x(0,t) = c \tilde p(0,t),& t\in (0,\infty),\\
     \tilde p_x(1,t) =0,& t\in (0,\infty)\\
     %\tilde p(x,0) = \tilde p_0(x),& x\in (0,1),\\
     u_{tt}=u_{xx},&(x,t)\in (0,1)\times(0,\infty),\\u(0,t)=0,& t\in (0,\infty), \\ u_x(1,t)=-(a-\tfrac{bd}{c})u_t(1,t)+b\tilde{p}(0,t),& t\in (0,\infty). %u(x,0) = u_0(x),\hspace{.1cm} u_t(x,0) = u_1(x),& x\in (0,1).
\end{cases}
\end{equation}
Following the same proof as we did earlier, we end up in a dead end due to the presence of the trace term $u_{tt}(1,t)$. To address this problem we have to consider more regularity in the initial data (in order to the trace to make sense), and at the moment the authors could not find an estimate that relates $u_{tt}(1,t)$ with the functional $V_1(u)$ defined in Section \ref{ReducedStab}. Also, the following system consisting on a slow heat equation coupled with a fast wave equation
\begin{equation}
\label{system1}
\begin{cases}
    \varepsilon^2 u_{tt}  = u_{xx}, & (x,t)\in (0,1)\times (0,\infty),  \\
    u(0,t) = 0, & t\in (0,\infty),\\
    u_x(1,t) = -  a\varepsilon u_t(1,t) + b p(0,t), & t\in (0,\infty),\\
   % u(x,0) = u_0(x),\: u_t(x,0) = u_1(x), & x\in (0,1),\\
    p_t  = p_{xx},& (x,t)\in (0,1)\times (0,\infty),\\
    p_x(0,t) = c p(0,t),& t\in (0,\infty),\\
    p_x(1,t) = d \varepsilon u_t(1,t),& t\in (0,\infty),
    %p(x,0) = p_0(x),& x\in (0,1),
\end{cases}
\end{equation}
has not been studied so far. This system may lead to some regularity issues to address, see for instance \cite[Section 4]{marxcerpa}, where a higher regularity is needed.
\appendix
\section{Some useful inequalities}
We state some useful inequalities that will help us through the article.
\begin{lemma}
\label{a1}
Let $w\in H^1(0,1)$, then we have that
\begin{equation*}
\abs{w(1)}^2 \leq 2 \abs{w(0)}^2 + 2 \int_0^1 \abs{w_x}^2 dx.
\end{equation*}
\end{lemma}
\noindent This can be proved by using the reverse triangular inequality, i.e., $\abs{w(1)}-\abs{w(0)}\leq \abs{w(1)-w(0)}$ together with the identity $w(1)-w(0)= \int_0^1 w_x dx$ and Jensen's inequalitiy.

\begin{lemma}
\label{trace_poincare}
Let $p\in H^1_L(0,1)$, then $\abs{p(1)}^2 \leq \norm{p_x}_{L^2(0,1)}$.
\end{lemma}
\begin{proof}
A straightforward application of \cite[Lemma 2.1]{guzma2020boundary}.
\end{proof}
\begin{lemma}
\label{trace_reduced}
Let $ (u,v) \in H^1_L(0,1) \times H^1(0,1)$. For every $\mu>0$, we have that 
$$-\mu\int_0^1 e^{\mu x} \abs{u_x+v}^2+ e^{-\mu x} \abs{u_x-v}^2 dx \leq -2\mu e^{-\mu} \abs{ u(1)}^2. $$
\end{lemma}
\begin{proof}
Note that
\begin{equation*}
\begin{aligned}
\int_0^1 e^{\mu x} \abs{u_x+v}^2+ e^{-\mu x} \abs{u_x-v}^2 dx & \geq \int_0^1 (u_x+v)^2+e^{-\mu}(u_x -v)^2 ,\\& \geq  e^{-\mu} \int_0^1 2( \abs{v}^2+\abs{u_x}^2) \geq 2e^{-\mu}\int_0^1\abs{u_x}^2dx,
\end{aligned}
\end{equation*}
multiplying by $-\mu$, and using Lemma \ref{trace_poincare} we get the desired result.
\end{proof}
The following lemma is taking from \cite[Chapter 2, Remark 2.2]{krstic2008boundary}
\begin{lemma}
\label{wirtinger}
Let $w$ be a continuously differentiable function, then we have
$$ \int_0^1 w^2 dx \leq \abs{w(i)}^2 + \frac{4}{\pi^2}\int_0^1 \abs{w_x}^2 dx, \hspace{.3cm} i=0,1.$$
\end{lemma}
%When $w(0)=0$ or $w(1)=0$ in Lemma \ref{wirtinger} we will refer to as Poincaré's inequality.

%%%%%%%%%%%%%%%%%%%%%%%%%%%%%%%%%%%%%%%%%%%%%%%%%%%%%%%%%%%%%%%%%%%%%%%%%%%%%%%%%%%%%%%%%%%%%%%%%%%%%%%%%%%%%%%%%%%%%%%%%%%%%%%

\bibliographystyle{plain}
\bibliography{references}

%%%%%%%%%%%%%%%%%%%%%%%%%%%%%%%%%%%%%%%%%%%%%%%%%%%%%%%%%%%%%%%%%%%%%%%%%%%%%%%%%%%%%%%%%%%%%%%%%%%%%%%%%%%%%%%%%%%%%%%%%%%%%%%

\end{document}